\documentclass[a4paper, reqno, 12pt]{amsart}
\usepackage[english]{babel}
\usepackage[T1]{fontenc}
\usepackage{mathptmx,amsmath}
\usepackage{hyperref}
\usepackage{graphicx}
\usepackage{fullpage}
\newtheorem{theorem}{Theorem}[section]

\newtheorem{lemma}[theorem]{Lemma}
\newtheorem{definition}[theorem]{Definition}
\newtheorem{proposition}[theorem]{Proposition}
 \theoremstyle{remark}

\numberwithin{equation}{section}

\def\cen{\centerline}
\def\n{\noindent}
\def\al{\alpha}

\def\va{\varphi}
\def\ep{\varepsilon}
\def\C{\mathbb C}
\def\R{\mathbb R}

\def\F{\mathcal F}
\def\ov{\overline}
\def\pa{\partial}
\def\psh{plurisubharmonic}
\def\ve{\varepsilon}
\def\de{\delta}
\def\nhd{neighborhood}

\def\De{\Delta}

\begin{document}
\title {Approximation  of  plurisubharmonic functions\\
 on  complex varieties}
\author{Nguyen Quang Dieu, Tang Van Long and Sanphet Ounheuan }
\address{Department of Mathematics, Hanoi National University of Education,
136 Xuan Thuy street, Cau Giay, Hanoi, Vietnam}
\email{  ngquangdieu@hn.vnn.vn, tangvan.long@gmail.com, sanphetMA@gmail.com }


%\keywords{Plurisubharmonic functions,  Projective pluripolar sets}

\date{\today}

%\dedicatory{}

\maketitle
\section{Introduction}
Let $V$ be a complex variety in a domain $D \subset  \mathbb C^ n (n  \ge 2)$ i.e, $V$ is a closed subset of $D$ and for every $z_0 \in V$, there exists a \nhd\ $U$ of $z_0$ such that $U \cap V$ is the common zero set of
holomorphic functions on $U.$
Let $PSH (V)$ denote the cone of \psh\ functions on $V.$
Recall  that a  function  $u:  V \to [-\infty, \infty)$ is  \psh\  if  $u$ is locally the restriction (on $V$) of  a \psh\ functions on an open subset of $D$. Notice that we regard the function identically $-\infty$ as plurisubharmonic.
A fundamental result of Fornaess and Narasimhan (cf. Theorem 5.3.1 in [FN]) asserts that an upper semicontinuous function $u: V \to \R \cup [-\infty, \infty)$ is \psh\  if and only if for every holomorphic map $\theta: \De \to V$, where $\De$ is the unit disk in $\C$, we have $u \circ \theta$ is subharmonic on $\De.$ This powerful result implies immediately the nontrivial facts that  plurisubharmonicity is preserved under uniform convergence.

We write $PSH (V)$ for the set of \psh\  functions on $V$. In this paper, we address the question of approximating elements of $PSH (V)$ which are bounded from above by \psh\  functions on $V$ which are continuous on $V$
or on $\ov {V}$. After subtracting constants, we need only consider functions in $PSH^{-} (V)$ the convex cone of  \psh\  functions $u$ on $V$ such that the upper semicontinuous regularization
$$u^* (\xi): =\varlimsup\limits_{z \to \xi, z \in V} u(z)<0,  \ \forall \xi \in \partial V.$$
Our  main tool is is a duality theorem of Edwards which  expresses upper envelopes of plurisubharmonic functions  taken in convex sub-cones of $PSH(V)$
in terms of  Jensen measures with respect to these cones.
This approach has been used  [Wik], [DW] and [Di]  in the case where $V$ is a bounded domain in $\C^n$.
The general principle is that the approximation of elements in a cone $\mathcal A \subset PSH^{-} (V)$ by elements in a smaller cone $\mathcal B$ is possible when we have equality of the set of Jensen measures with respect to $\mathcal A$ and $\mathcal B.$
Nevertheless, in our setting, there are at least two difficulties, first the standard local smoothing by convolving with approximate identities is not possible on the complex variety $V$, the second one stems from the fact that (the upper
semicontinuous regularization of) the supremum of a family of plurisubharmonic functions which are locally uniformly bounded from above is not necessarily plurisubharmonic (cf. Example 1.4 in [Ze]).
Therefore, we have to make extra conditions to overcome these obstacles. Namely, for the first one, the variety $V$ is assumed to be Stein i.e., there exists a \psh\ exhaustion  function on $V$, so that a global approximation
theorem of Fornaess and Narasimhan (cf. Theorem 5.5 in [FN]) is applicable (cf. Theorem \ref{thm2}), the other one is settled by putting some restriction on the part of $V$ which fails to be locally irreducible (cf. Theorem \ref{thm1}).

In order to formulate our results properly, it is convenient to introduce the following notions pertaining to our work.
\begin{definition}\label{def2}
Let $V$  be a  complex variety  in a domain  $D \subset  \C^n$. For  a  point  $  z \in \ov V$, we  define  two  classes  of  Jensen  measures
 $$
 \begin{aligned}
  &J_z = \{ \mu \in \mathcal{B} (\ov V): u(z) \le  \int\limits_{\ov V} ud\mu,  \forall  u \in PSH^{-}(V)\};\\
  &J^c_z = \{ \mu \in \mathcal{B} (\ov V): u(z) \le  \int\limits_{\ov V} ud\mu,  \forall  u \in PSH^{-}_c (V)\};
  \end{aligned}
 $$
 where $PSH^{-}_c (V)$ is the cone of negative continuous functions on $\ov{V}$  which is \psh\ on $V$ and
 $\mathcal{B} (\ov V)$ denotes  the  set of  Borel  probability  measures  with  support  on  $\ov V$.
\end{definition}
The connection between Jensen measures and approximation of plurisubharmonic functions stems from the following fact which is a simple consequence of Fatou's lemma.
\begin{proposition}\label{pro1}
Let $E$ be a subset of $V$ such that for every $u \in PSH^{-} (V)$, there exists a sequence $\{u_j\}_{j \ge 1} \subset PSH^{-}_c (V)$ having the following properties:

\n
(i) $u_j \to u$ pointwise on $E$.

\n
(ii) $\varlimsup\limits_{j \to \infty} v_j \le u$ on $V.$

Then $J_z=J^c_z$ for every $z \in E.$
\end{proposition}
In the opposite direction, the next result  gives a sufficient condition so that
pointwise approximation of negative \psh\ functions on complex varieties by continuous plurisubharmonic ones is possible.
Before formulating it, recall that a complex variety $V \subset D \subset \mathbb C^n$ is said to be {\it locally irreducible} if it is so at
 at every point $a \in V$ i.e., there is a  fundamental system of \nhd s $U_j$ of $a$ such that
$U_j \cap V$ is irreducible in $U_j$ for every $j.$ This point of local irreducibility was recorded incorrectly in Definition 1.1 of [Wik2] where it only requires a single \nhd\ $U$ of $a$ such that
$U \cap V$ is locally irreducible. In general, it is quite hard to decide whether $V$ is locally irreducible near its singular locus.
Under some restriction on the set of local irreducible points of $V$ we have the following result.
\begin{theorem}\label{thm1}Let $V$ be a complex variety in a bounded  domain $D \subset \C^n.$  Assume that there exists $\psi \in PSH^{-} (V), \psi \not \equiv -\infty$
satisfying the following conditions:

\n
(i) $F:=\{z \in V: \psi (z)=-\infty\}$ is a closed subset of $V.$

\n
(ii) $V \setminus F$ is a locally irreducible complex variety in $D \setminus F.$

Suppose that  there exists $E \subset V$ such that $J_z = J^c_z$ for all $z \in V \setminus E.$
Then for every $u \in PSH^{-} (V), u^*<0$, there exists two sequences $\{u_j\}_{j \ge 1} \subset PSH^{-} (V \setminus F)$ and $\{v_j\}_{j \ge 1} \subset  PSH^{-}_c (V)$
having the following properties:

\n
(a) $u_j \downarrow u$ on $V \setminus F$ and $u_j$ is continuous at every point of $V \setminus (E \cup F).$

\n
(b) $v_j \to u$ pointwise on $V\setminus (\ov E \cup F)$ and
$\varlimsup\limits_{j \to \infty} v_j \le u^*$ on $\ov V$.

\n
(c) Suppose in addition that  $V$ is of pure dimension $k, \ov {E}$ is pluripolar in $V$  and $u$ is  locally bounded on $V.$  Then the sequence $\{v_j\}_{j \ge 1}$ can be chosen to be locally uniformly bounded on $V$
and $(dd^c v_j)^k \to (dd^c u)^k$ on $U$ in the weak $^*-$topology as $j \to \infty$.
\end{theorem}
\n
We prefer to postpone a brief discussion on pluripolar sets and Monge-Amp\`ere operator on the complex variety $V$ in the next section.
Several remarks concerning Theorem \ref{thm1} are now in order.

\n
{\bf Remarks.} (a) We do not know if there always exists a function $\psi$ that satisfies (i) and (ii) even in the case where $V$ is Stein.
However,  if $V$ is a complex variety in a bounded pseudoconvex domain $D'$ that contains $\ov{D}$ then such a function $\psi$ can be found as follows.
Since the singular part  of $V$, denoted by $V_s$  is a (proper) complex subvariety of $V$, we deduce that $V_s$ is also a complex subvariety of $D'$.
Thus, $V_s=\{z \in D': f_1(z)=\cdots=f_k(z)=0\}$, where $f_1, \cdots, f_k$ are holomorphic functions on $D'$.
It is clear that $\psi:= \log(\vert f_1\vert+\cdots+\vert f_k\vert)-M \in PSH^{-} (V)$ is the desired plurisubharmonic function for $M>0$ large enough.

\n
(b) We should say that in the case where $V$ is an open subset of $\C^n$ and
$E =\emptyset$, the assertions (a) and (b) are implicitly contained in Theorem 3.1 of [DW]. We thought it may be of interest to include the case where there exists some {\it exceptional} set $E$
so that the two classes of Jensen measures may differ on $E$.

\n
(c) The statement (c) was essentially proved in Theorem of [Di]. Nevertheless, the somewhat complicated proof given there does not even generalize to the case $V$ is a smooth complex variety since it uses convolutions with smoothing kernels.
\vskip0,3cm
\n
The next main result deals with the case where the exceptional set $E$ mentioned in Theorem {\it might} occurs.
\begin{theorem}\label{thm2}
Let $V$ be a  Stein complex variety in a  bounded domain  $D \subset \C^n$. Suppose that there exists  $v \in PSH^{-} (V), v \not \equiv-\infty$  and a compact  $K\subset \partial  V$ satisfying the following properties:

\n
(i) $\underset{z \to \xi}\lim  v(z)   = - \infty,\ \forall   \xi \in K;$

\n
(ii) $J_\xi^c = \{\delta_\xi\},  \forall  \xi \in (\partial V)\setminus  K.$

Then the following statements hold true:

\n
(a) For every $z \in V \setminus E$, where $E:=\{z \in V:  v(z)=-\infty\}$ we have
 $J_z =  J^c_z.$

 \n
 (b) Suppose in addition that $V$ is locally irreducible, then for every $u \in PSH^{-} (V)$ there exists a sequence $u_j \in PSH^{-} (V)$ such that  $u_j$ is continuous at every point of
 $U:=\ov{V} \setminus (K \cup E)$ and $u_j^* \downarrow u^*$ on $U.$
\end{theorem}
\n
{\bf Remarks.} (a) The condition (ii) is fulfilled at $\xi \in \partial V$ if there is a local continuous \psh\ barrier at $\xi$ i.e.,
there exists $u \in PSH^{-}_c (V \cap \mathbb B (\xi, r))$ for some small  $\mathbb B(\xi, r) \subset \C^n$
such that $u(\xi)=0$ whereas $u<0$ elsewhere.
Indeed, by shrinking $r$ we may assume that $u<-\delta$ for $\ov{V} \cap \partial \mathbb B(\xi, r)$ for some $\de>0$.
By setting $\tilde u:=\max \{u, -\de\}$ on $\ov V \cap \mathbb B(\xi, r)$ and $\tilde u:= -\de$ out of $\mathbb B(\xi, r)$ and using the gluing lemma (cf. Lemma \ref{gluing})
we have $\tilde u \in PSH^{-}_c (V)$ and satisfies $\tilde u(\xi)=0, \tilde u<0$ elsewhere. It easily implies that $J_\xi=\{\delta_\xi\}.$
This reasoning is essentially contained in Proposition 1.4 in [Si].

\n
(b) Let $\va (z,w):=\vert z\vert^4+\vert w\vert^4, (z, w) \in \mathbb C^2$  and $\mathbb B_2$ be the unit ball in $\mathbb C^2.$ We set
$$D:=\{(z,w, t) \in \mathbb B_2 \times \mathbb C: \vert t\vert<e^{-\va (z,w)}\}, V:=\{(z, w, t) \in D: t^2=z^2w\}.$$
Then $V$ is a complex variety in the (bounded)  Hartogs pseudoconvex domain $D$.
Notice that, $V$ is locally reducible at every point $(0, w, 0) \in D$ with $0<\vert w \vert< 1$  (cf. [Ch] p. 56).
Given $\xi_0 =(z_0, w_0, t_0) \in \partial V, z_0w_0 \ne 0.$  By strict plurisubharmonicity of $\va$ at $(z_0, w_0)$ and strict pseudoconvexity of $\partial \mathbb B_2$ we see that every
Jensen measure $\mu \in J_{\xi_0}$ is supported at $\{\xi_0\} \cup\{(z_0, w_0, -t_0)\}.$
It follows that  $J_{\xi_0}=\{\de_{\xi_0}\}.$ Thus we may apply Theorem \ref{thm2} with
$$K:=\{(z, 0, 0): \vert z\vert=1\} \cup \{(0, w, 0): \vert w\vert =1\}, v(z, w):= \log \vert z w\vert$$
and get $E=\emptyset.$ The above construction is also inspired  from an example of Sibony (cf. [Si], p. 310).
\vskip0,3cm
\n
The theorem below deals with the problem of finding a bounded continuous {\it maximal} plurisubharmonic $u$ on $V$ such that the boundary values of $u$ coincides with a given continuous function defined on part of the boundary $\partial V.$
A weaker version of this result is given in Theorem 1.8  of [Wik2] under the assumption that $V$ admits a $B-$regular \nhd\  in $\mathbb C^n.$
Recall that $u \in PSH (V)$ is said to be maximal if for every relatively compact open subset $U$ of $V$ and every $v \in PSH(V)$ such that $v \le u$ on $V \setminus U$ we have $v \le u$ on $V.$
This definition is taken from Definition 1.6 in [Wik2] and is  analogous to the classical one given by Sadullaev for  the case where $V$ is an open set of $\mathbb C^n$ (cf. Proposition 3.1.1 in [Kl]).
\begin{theorem}\label{thm3}
Let $V$ be a Stein locally irreducible complex variety in a bounded domain $D \subset \C^n$. Suppose that there is $v \in PSH^{-} (V), v > -\infty$   on  $V$ and a compact  $K\subset \partial  V$ satisfying the following properties:

\n
(i) $\underset{z \to \xi}\lim  v(z)   = - \infty,\ \forall   \xi \in K;$

\n
(ii)  Every $\xi \in (\partial V) \setminus K$ admits a local continuous \psh\ barrier.

Then for every $\va \in \mathcal C(\partial V)$, there exists a unique bounded, maximal continuous  \psh\ function $u$ on $V$ such that
$$\lim_{z \to \xi, z \in V} u(z)=\va(\xi), \ \forall \xi \in (\partial V)\setminus K.$$
\end{theorem}
Our final result generalizes Theorem 2.3 in [Wik2] in that it does not assume the existence of a $B-$regular \nhd\  of $V$ in $\mathbb C^n.$
\begin{theorem}\label{thm4}
Let $V$ be a complex variety in  a bounded domain $D \subset \C^n$ having the following properties:

\n
(i) $J_\xi^c = \{\delta_\xi\},  \forall  \xi \in \partial V.$

\n
(ii) There exists a negative continuous \psh\ exhaustion function $\rho$ for $V$.

Then for every $u \in PSH^{-} (V)$ there exists a decreasing sequence  $\{u_j\}_{j \ge 1} \subset  PSH^{-}_c (V)$ such that
$u_j \downarrow u^*$ on  $ \ov{V}$.
\end{theorem}
\n
{\bf Remarks.} (a) According to Theorem \ref{thm2}(b), under the additional assumption
that $V$ is  Stein {\it locally irreducible} the conclusion of the theorem is still valid without assuming (ii).

\n
(b) The function $-1/\rho(z)+\vert z\vert^2$ is continuous  strictly \psh\ on $V$ and tends to $+\infty$ as $z \to \partial V.$ In particular, $V$ is Stein.

\n
(c) If $V$ satisfies $(ii)$ alone then every  $\mu \in J^c_\xi$ is supported on $\partial V$. This follows easily from the estimate $0=\rho(\xi) \le \int_{\ov{V}} \rho d\mu=\int_V \rho d\mu.$
\maketitle
\section{Preliminaries}
We recall the general version of Edward's duality theorem which says that upper envelopes of upper semicontinuous functions defined on compact metric spaces may be
expressed as lower envelopes of integrals with respect to certain classes of measures.

Let $X$ be a compact metric space, and let $\mathcal F$ be a convex cone of real-valued, bounded from above and upper semicontinuous functions on $X$ containing all the constants. If
$g$ is a real-valued function on $X$, then we define
$$Sg(z):= \sup \{u(z) : u \in \mathcal  F, u \le g\}.$$
Denote by $\mathcal B(X)$ the class of positive, regular Borel measures on $X$.
For $z \in X$ we define
$$J_z^{\mathcal F}:= \{\mu \in \mathcal B (X): u(z) \le \int_X ud\mu,  \ \forall u \in \F\}.$$
It is easy to see that $J_z^{\mathcal F}$  is a convex, weak-$*$ compact subset of $\mathbb B(X)$. Moreover, $\mu (X)=1$ for every $\mu \in J_z^{\mathcal  F}$
since $\mathcal F$ contains the constants.
On the other hand, if g is a Borel measurable function on X, then we set
$$Ig(z):= \inf\{\int_X gd\mu:  \mu \in J_z^{\mathcal F} \}.$$
\n
Now we are able to formulate the following basic duality theorem of Edwards (cf. [Ed], [Wik1]).
\begin{theorem}
Let $\F$ be as above. If $g$ is lower semicontinuous on X, then $Sg = Ig.$
\end{theorem}
Apparently the first use of Edwards's duality theorem in pluripotential theory has been made in the seminal work  [Si]  where we can find a systematic study of
domains in $\C^n$ on which the Dirichlet problem for plurisubharmonic functions is solvable.

In our context, by applying the above theorem to the convex cones $\mathcal F_1:=PSH^{-} (V) $
and $\mathcal F_2:=PSH^{-}_c (V)$ we obtain the following result which will be refereed to as Edward's duality theorem.
\begin{theorem}\label{thmEd} (Edward's theorem)
Let $\varphi: \ov V \to (-\infty, +\infty]$ be a lower semicontinuous function. Then  we have
$$\begin{aligned}
&\inf \Big \{\int\limits_{\ov V} \varphi d\mu,  \mu \in  J_z\Big \}=  \sup\{u (z): u \in PSH^{-}(V), u^* \le  \varphi \ \ \text{ on}\ \ \ov V \},\\
&\inf \Big  \{\int\limits_{\ov V} \varphi d\mu,  \mu \in  J^c_z \Big \}=  \sup\{u (z): u \in PSH^{-}_c (V), u \le  \varphi \ \ \text{ on}\ \ \ov V \}.
\end{aligned}$$
\end{theorem}
We will make a good use of the following result about gluing \psh\  functions on complex varieties. This fact has been also  used implicitly in the proof of Theorem 1.8 of [Wik2] .
\begin{lemma} \label{gluing}
Let $V$ be a complex variety of a domain $D \subset \C^n, U \subset V$ be an open subset and $u \in PSH (V), v \in PSH(U)$.
Assume that $\varlimsup\limits_{z \to \xi} v(\xi) \le u(z) \  \forall z \in \partial U.$ Then the function
$$w:=\begin{cases}
\max \{u, v\}  &  \text{on}\  U\\
u &  \text{on}\  V \setminus U.
\end{cases}$$
 belongs to $PSH(V).$
\end{lemma}
\begin{proof}
The proof uses again the above mentioned Fornaess-Narasimhan's criterion for membership in $PSH(V).$ More precisely, by the assumption $w$ is upper semicontinuous on $V$, so it remains to check that
$w \circ \theta$ is subharmonic on $\De$ for every holomorphic map $\theta: \De \to V.$ Clearly
$$w \circ \theta=\begin{cases}
\max \{u \circ \theta, v \circ \theta\}  &  \text{on}\  \theta^{-1} (U) \\
u \circ \theta &  \text{on}\  \De \setminus \theta^{-1} (U).
\end{cases}$$
Since $\theta^{-1} (U)$ is an open subset of $\De$ we see that if  $\lambda \in \partial (\theta^{-1}(U))$ then  $  \theta(\lambda) \in  \partial U,$  and hence
$$
\limsup\limits_{ \theta^{-1}(U)\ni t \to \lambda}  v\circ \theta(t) \le \limsup\limits_{U\ni z \to \theta(\lambda)} v(z) \le u \circ \theta (\lambda).
$$
Thus we may apply the usual gluing lemma for subharmonic functions to reach that $w \circ \theta$ is subharmonic on $\De.$
The proof is complete.
\end{proof}
The following fact about plurisubharmonicity of upper envelopes of plurisubharmonic functions on complex varieties which may not be locally irreducible is needed
in the proof of Theorem 1.3.
\begin{lemma} \label{lm3}
Let $V$ be a complex variety of a bounded domain $D \subset \C^n$.
Assume that there exists $\psi \in PSH^{-} (V), \psi \not \equiv -\infty$ satisfying the conditions (i) and (ii) in Theorem \ref{thm1}.
Then for every upper semicontinuous function $\va: V \to [-\infty, 0)$  we have $v=v^* \in PSH^{-} (V \setminus F)$, where
$$v(z):= \sup\{u (z): u \in PSH^{-} (V), u \le \va  \  \text{on}\  V\}, \ z \in V.$$
\end{lemma}
\begin{proof}
First, by the assumption on $\va$  we have $v<0$ on $V.$
This implies that $v^*$ is negative and plurisubharmonic on the regular part of $V.$
In view of (ii), we may apply the theorem on removable singularities for \psh\ functions (cf. Theorem 1.7 in [De])
 to conclude that $v^* \in PSH^{-} (V \setminus F)$.  Moreover, since $\va$ is upper semicontinuous, $v^* \le \va$ on $V.$

Next, for $\ve>0$ we set
$$v_\ve:=\begin{cases}
v^*+\ve \psi  & \ \text{on}\ \ V \setminus F\\
-\infty&  \text{on}\  F.
\end{cases}$$
Then $v_\ve \le \va<0$ and upper semicontinuous on $V.$ Moreover, by (i), for any holomorphic mapping $\theta: \De \to V$, the composition map $v_{\ve} \circ \theta$ is
subharmonic on the open set $\De \setminus \theta^{-1} (F)$ and equal to $-\infty$ on $\theta^{-1} (F)$.
It follows that $v_\ve \circ \theta$ is subharmonic entirely on $\De.$
So we may apply Fornaess-Narasimhan's criterion to conclude that $v_\ve \in PSH^{-} (V).$
Therefore, $v_\ve \le v$ on $V$. By letting $\ve \downarrow 0$ we obtain $v^* \le v$ on $V \setminus F.$ This finishes the proof
since the reverse inequality is clear.
\end{proof}
Now we turn to some basic notions of pluripotential theory on complex varieties which are involved in the statement of Theorem \ref{thm1}.
Let $V$ be a complex variety of pure dimension $k$ in a bounded domain $D \subset \mathbb C^n.$
According to Bedford in [Be] (see also [De] and [Ze]), the complex Monge-Amp\`ere operator
$$(dd^c )^k: PSH(V) \cap L^{\infty}_{\text{loc}} (V) \to M_{n, n} (V),$$
where $M_{n ,n} (V)$ denotes Radon measures on $V,$ may be defined in the usual way on the regular locus $V_r$ of $V$ (cf. [Kl] p.113), and it extends "by zero" through the singular locus $V_s$ i.e., for Borel sets
$E \subset V$
$$\int_E (dd\psi)^k: =\int_{E \cap V_r} (dd^c \psi)^k, \ \forall \psi \in PSH(V) \cap L^{\infty}_{\text{loc}} (V).$$
Following Bedford and Taylor (cf. Theorem 4.4.2 in [Kl]), this operator can be used to characterize maximality of locally bounded \psh\  functions on smooth complex varieties.
Moreover, we may use $(dd^c)^k$ to identify pluripolar subsets of $V$.
Recall that,  $X \subset V$ is called  {\it pluripolar} if for every $a \in X$ there exists a \nhd\  $U$ of $a$ in $V$ and $v \in PSH (U)$ such that $v \not \equiv -\infty$ and $v|_{X \cap U}=-\infty.$
For instance, the singular locus $V_s$ of $V$ is pluripolar (in $V$) being a {\it proper} complex subvariety of $V$.
It is well known that $(dd^c \psi)^k$ does not charge (Borel) pluripolar sets for locally bounded \psh\  functions $\psi.$
A major problem in pluripotential theory is to decide when a (locally) pluripolar set is globally pluripolar i.e., there exists $v \in PSH(V), v \not\equiv-\infty$ such that
$v|_V \equiv -\infty.$
Using again the operator $(dd^c)^k,$ Bedford proved that every locally pluripolar subset of $V$ is globally pluripolar provided that $V$ is Stein (cf. Theorem 5.3 in [Be]).
Note that in the case where $V$ is an open subset of $\mathbb C^n,$ this statement is a celebrated theorem of Josefson.

Our final ingredients consists of a few standard facts about upper semicontinuous and lower semicontinuous functions on compact sets of $\C^n$.
First, we have an elementary yet useful result of Choquet (cf. Lemma 2.3.4 in [Kl]).
\n
\begin{lemma}\label{lmC}
Let $\{u_\al\}_{\al \in \mathcal {A}}$ be a family of upper semicontinuous functions on $\ov{V}$ which is locally bounded from above. Then there exists a countable subfamily $\mathcal {B}$ of $\mathcal A$
such that
$$(\sup \{u_\al: \al \in \mathcal {B}\})^*=(\sup\{u_\al: \al \in \mathcal {A}\})^*.$$
 If $u_\al$ are lower semicontinuous then $\mathcal {B}$ can be chosen so that
 $$\sup \{u_\al: \al \in \mathcal {B}\}=\sup\{u_\al: \al \in \mathcal{A}\}.$$
\end{lemma}
The next two simple lemmas deal with sequences of upper and lower semicontinuous on compact sets of $\C^n$.
\begin{lemma}\label{lm1}
 Let $\{ f_j\}_{j \ge 1}$ is a decreasing sequence of upper semicontinuous functions defined on a compact $K \subset \C^n$ and $g$ be a lower semicontinuous continuous  function  on $K$ such that
 $$\lim_{j \to \infty} f_j (x) \le g(x),  \forall  x \in K.$$ Then for every
$\varepsilon  > 0$ there exists  $j_0$  such that  if $j \ge j_0$ then
$$f_j ( x)<g( x) +\varepsilon,\ \forall  x \in  K.$$
\end{lemma}
\begin{proof}
For $j \ge 1$, we let  $K_j: =\{x \in K: f_j (x)-g(x) \ge \varepsilon \}.$
By the assumptions, we infer that $\{K_j\}_{j \ge 1}$ is a decreasing sequence of compact sets such that $\bigcap_{j \ge 1} K_j =\emptyset.$
Thus we can find an index $j_0 \ge 1$ such that $K_j=\emptyset$ for $j \ge j_0.$
This proves our lemma.
\end{proof}
\begin{lemma}\label{lm2}
Let $X$  be a subset  of  $\C^n$  and  $\{\va_j\}_{j \ge 1}$  be a  sequence  of  lower semicontinuous  functions  on  $X$  that  increases  to  a lower  semicontinuous  function  $\va$ on $X$.
Then  for every  sequence  $\{a_j\}_{j \ge 1} \subset  X$  with  $a_j  \to  a \in X$ we have
$$\va (a) \le \varliminf\limits_{j\to \infty} \va_j (a_j).$$
\end{lemma}
\begin{proof} For $j \ge k$ we have  $\va_k (a_j) \le \va_j (a_j)$. By letting $j \to \infty$ and using lower semicontinuity of $\va_k$ at $a$ we obtain
$$\va_k (a) \le \varliminf\limits_{j\to \infty} \va_k (a_j) \le \varliminf\limits_{j\to \infty} \va_j (a_j).$$
The desired conclusion follows by letting $k \to \infty$ in the right hand side.
\end{proof}

\section{Proofs of the main results}
\begin{proof}( {\it of Proposition \ref{pro1}})  Obviously $J_z \subset J^c_z, \forall z \in V.$ Conversely,
fix $z \in E$ and $\mu \in J^c_z.$ For every $u \in PSH^{-} (V)$ we choose a sequence $\{u_j\}_{j \ge 1} \subset PSH^{-} (V)$ that satisfy the conditions (i) and (ii).
Then we have
$$u_j (z)\le \int_V u_j d\mu, \ \forall j \ge 1.$$
By letting $j \to \infty$ and making use of Fatou's lemma we get
$$u(z)\le \int_V ud\mu.$$
Thus $\mu \in J_z$ as desired.
\end{proof}
\begin{proof} ({\it of Theorem \ref{thm1}})
We will prove (a) and (b) simultaneously.
 Fix $u \in PSH^{-} (V)$.
We now follow closely the arguments in Theorem 3.1 of  [DW].
Choose a sequence $\{\va_j\}_{j \ge 1} \subset \mathcal C (\ov V)$ with $\va_j \downarrow u^*$ on $\ov V.$
For every $j \ge 1$ we define $S \va_j$ and $S^c \va_j$ as follows
\begin{equation}\label{envelop}
\begin{aligned}
&S\va_j (z):=  \sup\{u (z): u \in PSH^{-}(V), u^* \le \varphi_j \ \ \text{ on}\ \ \ov V\}, \  z \in V, \\
&S^c \va_j (z):=  \sup\{u (z): u \in PSH^{-}_c (V), u \le  \varphi_j \ \ \text{ on}\ \ \ov V\}, z \in \ov{V}.
\end{aligned}
\end{equation}
Then from Edwards' theorem and the assumption that $J_z=J^c_z$ for every $z \in V \setminus E$
we obtain
$$S^c \va_j =S\va_j \  \text{on}\   V \setminus E.$$
Since $\va_j \in \mathcal C(\ov V)$, we infer that $(S\va_j)^* \le \va_j$. Under the assumptions (i) and (ii), we may apply Lemma \ref{lm3} to obtain
$$u_j:=S\va_j=(S\va_j)^* \in PSH^{-} (V \setminus F) \ \forall j \ge 1.$$
On the other hand, since $S^c \va_j$ is lower semicontinuous on $V$ we deduce that $(S \va_j)^*$ is continuous at every point in $V \setminus (E \cup F).$
Furthermore, we observe that $u \le S \va_j  \le \va_j$ on $V$ for every $j.$ Therefore $u_j=S\va_j \downarrow u$ on $V$.
Thus we get the assertion (a) of the theorem.

Next, we let $\{K_j\}_{j \ge 1}$ be an exhaustion of $V':=V \setminus ( \ov{E} \cup F)$ by compact subsets. For every $j \ge 1,$ by Choquet's topological lemma \ref{lmC}, we can find a sequence
$\{v_{l, j}\}_{l \ge 1} \subset PSH^{-}_c (V)$ that increases to $S^c \va_j$ on $\ov {V}$. By Dini's theorem and continuity of $S \va_j$ on $V'$, the convergence is uniform on $K_j$ as $l \to \infty$.
Thus we can choose $v_{l(j), j} \in PSH^{-}_c (V)$ such that
$$\Vert S\va_j -v_{l(j), j} \Vert_{K_j} \le 1/j, v_{l(j), j} \le \va_j \  \text{on}\  \partial V.$$
It is easy to check that $v_j:=v_{l(j), j}$ converges pointwise to $u$ on $V'$ and
$$\varlimsup\limits_{j \to \infty} v_j^* \le \lim_{j \to \infty} \va_j=u^* \   \text{on}\  \ov V.$$
\n
(c)
 Under the additional assumptions on $V, E$ and $u$, we will first show that the above construction of $\{v_j\}_{j \ge 1}$ can be modified so that $v_j$ is  locally uniformly bounded on $V$
and that $v_j$ converges to $u$ {\it locally in capacity} on $V.$
The desired conclusion on weak $^*-$ convergence of Monge-Amp\`ere measures would then follows from a result of Xing (cf. Theorem 1 in [Xi]).

Return to the proof, we let $\va \in PSH (V), \va \not \equiv -\infty$ such that $\va|_{\ov E} =-\infty$. Let $\{V_j\}_{j \ge 1}$ be an exhaustion of $V$ by  relatively compact open subsets.
For $j \ge 1$ we define the open set $W_j:=\{z \in V: \va(z)<-j\}.$
Since $u$ is bounded from below on $V_j$ for every $j$, we infer that the sequence $\{\va'_j\}_{j \ge 1} \subset \mathcal C(\ov{V})$ defined by
$$\va'_j: =\max \{\va_j, \al_j\}, \ \text{where}\ \ \al_j =\inf_{V_j} u, \ \forall j \ge 1,$$
decreases to $u$ on $V$ as well. Define the envelopes $S^c \va'_j$ and $S \va'_j$ as in (\ref{envelop}). Then we have
$$S^c \va'_j=S\va_j  \ \text{on}\  V \setminus \ov{E}.$$
Choose a sequence $\{v_{l, j}\}_{l \ge 1} \subset PSH^{-}_c (V)$ that increases to $S^c \va_j$ on $\ov {V}$ such that $v_{l, j} \ge \al_j$ on $V_j$ for every $l \ge 1.$
By Dini's theorem and continuity of $S\va_j$ on $V \setminus \ov{E}$
we can choose $l(j) \ge 1$ such that
\begin{equation} \label{vj}
\Vert S \va'_j -v_{l(j), j} \Vert_{\ov{V_j } \setminus W_j} \le 1/j, v_{l(j), j} \le \va_j \  \text{on}\  \partial V.
\end{equation}
It is easy to check that the sequence $v_j:=v_{l(j), j}$ is locally uniformly bounded on $V$ and satisfies the conditions given in (b).
Next, fix $z_0 \in V$, we must show that there exists a small \nhd\ $U$ of $z_0$ such that
$(dd^c v_j)^k$ converges weakly to $(dd^c u)^k$ on $U$. Consider two cases.

\n
{\it Case 1.} $z_0 \in V_r$.  Choose a \nhd\ $U$ of $z_0$ in $V$ which is biholomorphic to an open subset of $\C^k.$
It suffices to show $(dd^c v_j)^k$ converges to $(dd^c u)^k$ on $U.$ For simplicity of exposition, we may assume  $U \subset \C^k.$
Fix $\ve>0$ and a relatively compact Borel subset $W$ of   $U.$ We claim that
\begin{equation} \label{cap}
\lim_{j \to \infty} C(\{z \in W: \vert v_j (z)-u(z)\vert>\ve\}, U)=0.
\end{equation}
Recall that for a Borel subset $X$ of $U$, the relative capacity (or Bedford-Taylor capacity)
$C(X, U)$ is defined as
$$C(X, U):= \sup \Big \{\int_X (dd^c v)^k: v \in PSH (U), -1<v<0\Big \}.$$
Using Bedford-Taylor's theorem on quasi-continuity of $u$ (cf. Theorem 3.5.5. in [Kl]), there exists a compact set $F \subset U$ such that $u|_{F}$ is continuous whereas $C(U \setminus F, U)<\ve.$
Choose $j_0$ so large that
$C(U \cap W_{j_0}, U)<\ve.$
Since $S\va'_j \downarrow u$ on $F \setminus W_{j_0} $ and since $S\va'_j$ is continuous on $F \setminus W_{j_0}$ for $j >j_0$, by Dini's theorem $S\va'_j$ converges uniformly to $u$ on $F \setminus W_{j_0}.$
Combining this with  (\ref{vj}) we see that if $j$ is sufficiently large then
$$\{z \in W: \vert v_j (z)-u(z)\vert>\ve\} \subset W \cap W_{j_0}.$$
The claim (\ref{cap}) follows from the choice of $j_0.$
By applying the above mentioned theorem of Xing to $\{v_j\}_{j \ge 1}, u$ and the open set $U$ we conclude that $(dd^c v_j)^k$ converges weakly to $(dd^c u)^k$ on $U$.

\n
{\it Case 2.} $z_0 \in V_s$. Let $U$ be a \nhd\ of $z_0$ which is relatively compact in $V.$ Since $\{v_j\}_{j \ge 1}$ is uniformly bounded on $U,$
the Chern-Levine-Nirenbeg's inequality (cf. Proposition 3.4.2 in [Kl]) implies that the measures $(dd^c v_j)^k$ has uniformly bounded masses on compact sets of $U \setminus V_s$.
On the other hand, by Lemma 3.1 in [Be], the set $U \cap V_s$ has outer capacity zero in $U.$ It follows that $(dd^c v_j)^k$ is uniformly bounded on compact sets of $U.$
Let $\mu \in M_{n, n} (U)$ be a cluster point of the sequence $\{(dd^c v_j)^k\}_{j \ge 1}$ in the weak $^*-$ topology. By the forgoing case, we know already that $\mu=(dd^c u)^k$ on $V_r \cap U.$
Moreover, the vanishing of the outer capacity of $U \cap V_s$ also yields that $\mu=(dd^c u)^k=0$ on $V_s \cap U.$
Therefore  $\mu=(dd^c u)^k$ on $U$. Thus $(dd^c v_j)^k$ converges to $(dd^c u)^k$ on $U$  in the weak $^*-$ topology.

The proof is thereby completed.
\end{proof}
\begin{proof} ({\it of Theorem \ref{thm2}})
(a) Obviously,  $ J_z \subset  J_z^c, \forall  z \in  V$.
 So it is enough to show the reverse conclusion.  Fix  $ z_0 \in V\setminus E$,  a measure
 $\mu \in J^c_{z_0}$ and  $u \in PSH^{-}(V)$. Since  $V$ is  Stein, by Fornaess-Narasimhan's approximation theorem there exist sequences  $\{u_j\}_{j \ge 1}, \{v_j\}_{j \ge 1} \subset PSH (V) \cap \mathcal C(V)$ with  $u_j \downarrow u$  and $v_j \downarrow v$ on  $V.$ Moreover, by the upper semicontinuity of  $u$  on  $\ov V$ we can  find  a  sequence
 $\{\va_j \}_{j \ge 1} \subset  \mathcal C(\ov  V)$  such that  $\va_j<0, \va_j  \downarrow u$ on  $\ov  V$.
 For  each $ j\ge  1$, we define the upper envelope
 $$S^c \va_j (z)= \sup \{u: u \in PSH^c (V), u \le  \va_j \ \text{on}\ \ov V\},  z \in \ov{V}.$$
Then
 \begin{equation}\label{eq1}
  S^c {\va_j}\le  \va_j  \ \text{ on}\ \     \ov V.
 \end{equation}
 By  Edward's  theorem \ref{thmEd}  and the hypothesis (ii)  we obtain
\begin{equation}\label{eq2}
S^c\va_j  = \va_j \  \text{on}\ \     (\partial  V) \setminus  K.
\end{equation}
 By Choquet's topological lemma \ref{lmC}, there exists  an increasing sequence
 $\{\va_{j,k}\}_{k \ge 1} \subset  PSH^c (V)$  such that
 \begin{equation}\label{eq3}
 \va_{j,k} \uparrow S^c \va_j \ \text{on}\ \ \ov V \ \text{ as} \ \  k \to \infty.
\end{equation}
Let $\{V_j\}_{j  \ge 1}$ be an exhaustion of $V$ by relatively compact  open subsets.
Fix $\ep >0$  and integers  $j \ge 1, p \ge 1$, we claim  that  there exist an index $k(j) > j$  such that for every $l \ge k(j)$ we have
\begin{equation}\label{eq4}
\va_{j, l} \ge  u  + \ep  v -\frac \ep2\ \text{on}\ \  \pa V_l.
 \end{equation}
 Indeed, suppose that  (\ref{eq4}) is false. Then  there  exists  a  sequence  $\{z_{k_l}\}_{l \ge 1} \subset V$ with $z_{k_l} \in \partial V_{k_l}$ such that
  $$\va_{j, k_l}(z_{k_l}) <u(z_{k_l})  + \ep  v(z_{k_l}) -\frac \ep2.$$
  By passing to a subsequence if necessary, we  can assume that  $z_{k_l} \to  z^* \in \pa  V$  as  $l \to  \infty$.
  In view of  Lemma  \ref{lm2} and the upper semicontinuity  of $u$  we get
  \begin{equation}\label{eq5}
S^c {\va_j} (z^*) \le  u(z^*) + \ep v^* (z^*) - \frac\ep2.
 \end{equation}
 Note that  the left hand side of  (\ref{eq5}) is bounded  from below by   $\inf_{\pa  V} \va_j > -\infty$, so $ z^* \notin  K$ by the assumption  (i).  Then it follows  from  (\ref{eq2}) and (\ref{eq5}), that
 $$\va_j(z^*) \le  u(z^*) + \ep  v(z^*) - \frac\ep2 <  u(z^*).$$
This  contradicts the fact that $\va_j \downarrow u$. Hence the claim (\ref{eq4}) is proved.

Now, from (\ref{eq4}) and Lemma \ref{lm1} we infer  that  for every $l \ge k(j)$ there exists an index
$m(l)  \ge  k(j)$ such that
$$\va_{j, l} \ge   u_{m(l)} + \ep v_{m(l)} -\ep\ \ \text{on}\ \ \partial V_l.$$
Using Lemma \ref{gluing} we see that the  function
$$
\tilde{u}_{j, p, l}:=\begin{cases}
\max\{\va_{j, l}, u_{m(l)} +  \ep  v_{m (l)} -\ep\} & \ \text{on}\ \ V_{l}\\
\va_{j, l}&  \text{on}\ \ov V \setminus  V_{l}
\end{cases}$$
belongs  to $PSH^{c}(V).$ Furthermore, it follows from  (\ref{eq3}) that
 $$ \tilde{u}_{j,p,l} \to \theta_{j, p}:=
\begin{cases}
  \max\{S^c {\va_{j}}, u + \ep  v -\ep\} & \text{ on}\ \   V\\
 S^c {\va_j}  & \ \text{on}\ \ \pa V
\end{cases}
$$
 as  $l \to  \infty$. In particular, since $v<0$ on $V$, by (\ref{eq1}) we conclude that
\begin{equation}  \label{eq6}
u +\ep  v -\ep \le \theta_{j, p}  \ \text{on}\ V,  \theta_{j, p} \le \va_j  \   \text{on}\  \ov{V}. \end{equation}
 Using  $\mu  \in  J_{z_0}^c$  we  obtain
 $$ \tilde{u}_{i, p. l}(z_0) \le  \int\limits_{\ov V} \tilde{u}_{j, p, l} d\mu =  \int\limits_{ V} \tilde{u}_{j, p, l} d\mu + \int\limits_{\pa V} \tilde{u}_{j, p, l} d\mu. $$
  By letting  $l \to \infty$,  using  Fatou's lemma, by  (\ref{eq6})  and  noting that  $ v< 0$ on $V$  we obtain the following estimates
   \begin{equation} \label{eq7}
  u(z_0) + \ep  v(z_0) -\ep \le \theta_{j,\ep}(z_0) \le  \int\limits_{\ov V} \va_j d\mu.
  \end{equation}
  Letting   $j \to \infty$ and then  $\ep \to 0$  (taking into account that $v(z_0)>-\infty$)
applying Fatou's lemma again we get
  $$ u(z_0) \le   \int\limits_{ \ov V} ud\mu .$$
This means that $\mu  \in  J_{z_0},$ we are done.

\n
(b) Let $\{\va_j \}_{j \ge 1} \subset  \mathcal C(\ov  V)$ be the sequence chosen in (a). For each $j \ge 1$ we define the envelope
$$S \va_j (z)= \sup \{u(z): u \in PSH^{-} (V), u \le  \va_j \ \text{on}\ \ov V\},  z \in V.$$
It follows from (a) and the {\it proof} of Theorem \ref{thm1} (a) that  $u_j:= S\va_j \in PSH^{-} (V)$
and the sequence $u_j \downarrow u$ on $V \setminus E$ as $j \to \infty$. Moreover, $u_j$ is continuous at every point of $V \setminus E.$
For boundary behavior of $u_j$ we fix $\xi \in (\partial V) \setminus K.$
Then we have
$$\va_j (\xi) =S^c \va_j (\xi) \le  \varliminf\limits_{z \to \xi, z \in V} S^c \va_j (z) \le  \varliminf\limits_{z \to \xi, z \in V} u_j (z).$$
Therefore $\lim_{z \to \xi} u_j (z)=\va_j (\xi).$ Thus $u_j$ is continuous at every point of $U:=\ov V \setminus (K \cup E)$ and
$u_j ^* \downarrow u^*$ on $U.$
The proof is complete.
\end{proof}
For the proof of the next theorem we require the following fact.
\begin{lemma} \label{alpha}
Let $V$ be a complex variety in a bounded domain $D \subset \mathbb C^n$ and $\xi \in \partial V$ be a boundary point.
Assume that there is a local continuous \psh\  barrier at $\xi.$
Then for every sequence $\{\va_j\}_{j \ge 1} \subset \mathcal C(\partial V), \va_j<0$ that decreases to an upper semicontinuous function $\va$ on $\partial V$
and every sequence $\{\xi_j\} \subset V$ with $\xi_j \to \xi$  we have
$$\varlimsup\limits_{j \to \infty} S^c \va_j (\xi_j)  \le  \va (\xi),$$
where $S^c \va_j (z):= \sup\{u(z): u \in PSH^{-}_c (V), u \le \va_j \ \text{on}\  \partial V\}, z \in \ov{V}.$
\end{lemma}
\begin{proof} Let $u$ be a local continuous \psh\ barrier at $\xi.$
By the argument given in the remark following Theorem \ref{thm2}, we may extend $u$ to $\tilde u \in PSH^{-}_c (V)$  such that $\tilde u$ is a barrier at $\xi.$
Let $\{\mu_j\}_{j \ge 1}, \mu_j \in J_{z_j}$ be a sequence of Jensen measures with compact support
in $\partial V.$ We claim that $\mu_j$ converges to $\delta_\xi$ in the weak $^*-$ topology.
It suffices to show that any cluster point of this sequence coincides with $\delta_\xi$. Let
$\mu^*$ be such a cluster point. Then we have
$$0=\varliminf\limits_{j \to \infty} \tilde u(z_j) \le \varliminf\limits_{j \to \infty} \int_{ V} \tilde ud\mu_j \le \int_{V}  \tilde ud\mu^* \le 0.$$
So $\mu^* =\delta_\xi.$ The proves the claim.
It follows, since $\va_j \downarrow \va$ on $\partial V,$ that
$$\varlimsup\limits_{j \to \infty} S^c \va_j (z_j) \le \varlimsup\limits_{j \to \infty} \int_{\partial V} \va_j d\mu_j \le \lim_{j \to \infty} \int_{\partial V} \va_j d\delta_{\xi}=\va(\xi).$$
This is the desired conclusion.
\end{proof}
\begin{proof}  ({\it of Theorem \ref{thm3}})
We split the proof in two two parts.

\n
{\it Existence.} After subtracting a large constant we may assume $\va<0$ on $\partial V.$ Define the upper envelopes
$$\begin{aligned}
&S\va (z):=  \sup\{u (z): u \in PSH^{-}(V), u \le \varphi \ \ \text{ on}\ \ \partial V\}, \  z \in V; \\
&S^c \va (z):=  \sup\{u (z): u \in PSH^{-}_c (V), u^* \le  \varphi \ \ \text{ on}\ \  \partial V\}, z \in \ov{V}.
\end{aligned}$$
In view of the assumption (a) and the remark following Theorem \ref{thm2} we have $J^c_\xi=\{\de_\xi\}$ for every $\xi \in  ( \partial V) \setminus K.$
So by Edwards' duality theorem (with $\va:=+\infty$ on $V$) we obtain
\begin{equation}\label{i1}
S^c \va =\va \  \text{on}\ ( \partial V) \setminus K.
\end{equation}
Furthermore, since $V$ is Stein, by Theorem 1.4 (a) we get $J_z=J^c_z$ for every $z \in V.$ So using again Edwards' duality theorem as above
we get
\begin{equation}
S \va =S^c \va \ \text{on}\  V.
\end{equation}
Since $V$ is locally irreducible, $u:=(S^c \va)^* \in PSH^{-} (V).$ Moreover, by  the assumption (ii) we may apply Lemma \ref{alpha} to $\va_j=\va$
to obtain
\begin{equation} \label{i2}
\varlimsup\limits_{z \to \xi, z \in V} u(z) \le \va(\xi), \  \forall \xi \in (\partial V)\setminus K.
\end{equation}
Fix $\ve>0$ and set $u_\ve:= u+\ve v.$
Then we infer from the last inequality and the assumption that $(u_\ve)^* \le \va$ on $\partial V.$ This implies that $u_\ve \le S \va$ on $V.$
By letting $\ve \downarrow 0$ and noting that $v>-\infty$ on $V$ we get
$u=S \va=S^c \va$ on $V$. Hence $u$ is lower semicontinuous on $V$, so $u \in PSH(V) \cap \mathcal C(V)$  and $\Vert u\Vert_V \le \Vert \va\Vert_{\partial V}.$
Next we show that $u$ has the right boundary values. Indeed, fix $\xi \in (\partial V) \setminus K$ and a sequence $\xi_j \to \xi, \xi_j \in V.$
By lower semicontinuity on $\ov{V}$ of $S^c \va$ and (\ref{i1}), (\ref{i2}) we have
$$\va (\xi) =S^c \va(\xi) \le \varliminf\limits_{j \to \infty} u(\xi_j) \le  \varlimsup\limits_{j \to \infty} u(\xi_j) \le \va(\xi) .$$
It follows that $u(z) \to \va(\xi)$ as $z \to \xi, z \in V.$  Finally, we let $w \in PSH (V)$ with $w \le u$ on $V \setminus U$ for some open set $U$ relatively compact in $V.$
Then by the gluing lemma \ref{gluing}
$$\tilde u(z):= \begin{cases}
\max \{u(z),  w(z)\},  & z \in U \\
u(z) & z \in V \setminus U
\end{cases}$$
 belongs to $PSH (V)$. Moreover, $\tilde u$ is a member in the defining family for $S \va$. Therefore
 $\tilde u \le S \va =u$ on $V$. In particular, $w \le u$ on $U.$  This proves maximality of $u$ and also completes the proof of the existence of the solution.

 \n
 {\it Uniqueness.} Assume that there exist bounded continuous \psh\ functions $u_1, u_2$ on $V$ such that
 $$\lim_{z \to \xi, z \in V} u_1(z)=\lim_{z \to \xi, z \in V} u_2 (z)=v(\xi), \ \forall \xi \in (\partial V)\setminus K.$$
 Let $\{V_j\}$ be a sequence of relative compact open subset of $V$ with $V_j \uparrow V.$
Fix $\ve>0$, since $u_2$ is bounded from below on $V$
we can find $j_0 \ge 1$ so large such that
 $$u_1+\ve v \le u_2+\ve \ \text{on}\  V \setminus  V_{j _0}.$$
 It follows  from maximality of $u_2$ that
 $$u_1+\ve v \le u_2+\ve \ \text{on}\   V.$$
 By letting $\ve \downarrow 0,$ we infer that $u_1 \le u_2$ on $V$. Similarly we also get $u_2 \le u_1$ on $V$. Therefore $u_1=u_2$ on $V.$

 The theorem is proved.
 \end{proof}

 %%%%%%
\begin{proof} ({\it of Therem \ref{thm4}})
 For  each  $j \ge  1$, put
 $$V_j := \{ z \in  V:  \rho(z) < - \frac1{2j^2}\}.$$
 Then  $V_j$ is  a relative compact  open subset of $V, \forall j \ge  1$ and  $V_j \uparrow V$. By the upper semicontinuity  of $u^*$ on  $ \ov V$ there  exists  a  sequence   $\varphi_j \in  \mathcal C( \ov  V), \va_j<0$ such that  $\varphi_j  \downarrow  u^*$  on  $ \ov V$.
Fornaess-Narasimhan's approximation theorem yields a sequence $\{v_j\}_{j \ge 1} \subset  PSH(V)  \cap   \mathcal C(V)$ such that  $v_j \downarrow u$ on  $V$.
 Now for $j \ge 1$ we define the upper envelopes
 $$\begin{aligned}
&S\va_j (z):=  \sup\{u (z): u \in PSH^{-}(V), u^* \le \varphi_j \ \ \text{ on}\ \ \ov V\}, \  z \in V \\
&S^c \va_j (z):=  \sup\{u (z): u \in PSH^{-}_c (V), u \le  \varphi_j \ \ \text{ on}\ \  \ov V\}, z \in \ov{V}.
\end{aligned}$$
Using Edwards' duality theorem  and the assumption (i) we obtain
\begin{equation} \label{inq1}
S^c \va_j =\va_j \  \text{on}\  \partial V.
\end{equation}
Furthermore, by Theorem \ref{thm2} (a) we get $J_z=J^c_z$ for every $z \in V.$ So by applying again Edwards' duality theorem as above
we get
\begin{equation} \label{inq2}
u \le S \va_j =S^c \va_j  \ \text{on}\  V.
\end{equation}
Next, fix $j \ge 1$, in view of Choquet's topological lemma \ref{lmC} there is a sequence $\{\va_{j, l}\}_{l \ge 1} \subset PSH^{-}_c (V) $ that increases to $S^c \va_j$ on
$\ov{V}$.
So, using (\ref{inq1}), (\ref{inq2}) and Lemma \ref{lm1} we can find $k(j)$
such that
\begin{equation} \label{inq4}
u \le  \va_{j, k(j)} + \frac1{3j} \ \text{on}\  \ov{V_j}
\end{equation}
and
\begin{equation} \label{inq5}
\max \{u^*, \va_j-\frac1{j}\}\le \va_{j, k(j)} \le \va_j \ \text{on}\  \partial V.
\end{equation}
Hence, it follows from Lemma \ref{lm1} and (\ref{inq4})  that  there exists  $l(j) \ge j$  satisfying $$v_{l(j)} \le  \va_{j, k(j)} + \frac1{2j} \ \text{on}\   \ov V_j.$$
Now we consider the new function
\begin{equation*}
 \tilde{u}_j: =
 \begin{cases}
 \max\big \{j\rho +\va_{j, k(j)}, v_{l(j)}-\frac1j \big \}   & \text{on}\   V_j\\
 j\rho  + \va_{j, k(j)}  &\text{on}\  \ov V \setminus  V_j.
 \end{cases}
  \end{equation*}
Observe that,  on  $\partial  V_j$ we have
 $$v_{l(j)} - \frac1j = \big(v_{l(j)} - \frac1{2j}\big) - \frac1{2j}  \le  \va_{j, k(j)} - \frac1{ 2j} = j\rho   + \va_{j, k(j)}.$$
 Therefore, by the gluing lemma \ref{gluing} $ \tilde{u}_j \in PSH^c (V)$ and  $ \tilde{u}_j|_{\partial  V} =  \va_{j, k(j)}$. This implies that  $\tilde u_j \to u^*$ pointwise on $\ov V$ as $j \to \infty$.
 Furthermore, since $\rho<0$ on $V$ the following estimates hold on $\ov {V}$
 \begin{equation} \label{inq6}
 \tilde u_j  \le \va_{j, k(j)}+\frac1{j}.
 \end{equation}
  Fix $j\ge  1$, for  each  $p \ge  j$ we set
 $$h_{p,j}:=  \sup\limits_{ j\le m\le p} \tilde{u}_m\ \  \text{and}\ \  u_j   = \sup\limits_{j \le m} \tilde{u}_m.
 $$
It is then clear that  $h_{p,j} \in PSH^c (V)$ and  $h_{p,j} \uparrow  u_j$ on $\ov{V}$. Moreover, by the choice of $k(j)$ we have
$$\va_j-\frac1{j} \le \va_{j, k(j)} \le u_j \le \va_j  \ \text{on}\  \partial V.$$
This implies that $u_j \downarrow  u^*$ on  $\ov V.$
We claim that  $h_{p,j}$ is uniformly  convergent  to  $u_j$  on $\ov{V}$ as  $p \to  \infty$.
Assume  otherwise,  there  exist
$\ep > 0$ and  a  sequence   $\{z_p \} \subset \ov{V}$  satisfying
\begin{equation}\label{eqa}
 h_{p,j}(z_p) +\ep <u_j (z_p),\ \forall  p \ge  j.
 \end{equation}
Since $u_j$ is lower semicontinuous on $\ov{V},$ we may assume $\{z_p\}_{p \ge 1} \subset V.$
 By the definition of $u_j$ and since \ref{eqa}, there exists $m(p) > p$  such that
\begin{equation}\label{eqc}
 h_{p,j}(z_p)  + \frac\ep2<  \tilde{u}_{m(p)}( z_p) ,\ \forall  p \ge  j.
 \end{equation}
 After switching  to a subsequence if necessary, we  can assume that $z_p \to z^* \in \ov{V}.$
We consider two cases.

 \n
{\it Case 1.} $z^* \in V.$ Then we may find $p_0 \ge 1$ such that
$\{z_p\}_{p \ge 1} \cup  \{z^*\}  \subset V_{m(p_0)}.$
By the definition of $\tilde u_j$ and taking into account the fact that $\rho<0$ on $V$ we obtain
 $$\tilde{u}_{m(p)}= v_{l(m(p))} - \frac1{m(p)} \  \text{on}\   V_{m(p)},$$
 for all $p \ge p_0$ large enough.
 This implies that
 $$\begin{aligned}
 \varlimsup\limits_{p \to \infty}\tilde{u}_{m(p)} (z_p)
 & =  \varlimsup\limits_{ p \to \infty} v_{l(m(p))} (z_p)
 & \le  \varlimsup\limits_{ p \to \infty} v_p(z_p)  \le  u(z^*),
 \end{aligned}
 $$
 where  the last  inequality  follows from  Lemma \ref{lm2}.
 Moreover, the left hand side of (\ref{eqc}) is  bounded from below by  $\inf\limits_{\ov V} \tilde{u}_j + \frac\ep2>-\infty$, therefore
 $u(z^*) > -\infty$.

On the other hand, using Lemma \ref{lm2} again we obtain
$$\begin{aligned}
\varliminf\limits_{ p \to \infty} h_{p,j} (z_p)
& \ge  u_j(z^*) \ge   \varliminf\limits_{ p \to \infty} \tilde{u}_{m(p)}(z^*) =  \varliminf\limits_{ p \to \infty} (v_{l(m(p))}(z^*) - \frac1{m(p)}) = u(z^*).
\end{aligned}
$$
Thus, by letting $p \to \infty$ in (\ref{eqc}) and taking limsup in both sides we get a contradiction.

\n
{\it Case 2.} $z^* \in \pa V$. It follows from (\ref{eqc}) and (\ref{inq6}) that
$$\begin{aligned}
\tilde u_j (z_p)+\frac{\ep}2
&
\le  h_{p,j}(z_p) + \frac{\ep}2< \tilde u_{m(p)} (z_p)
\le \va_{m(p), k(m_p)}(z_p)+\frac1{m(p)} \\
&\le S^c \va_{m(p)} (z_p) +\frac1{m(p)}
 \le \va_{m(p)}(z_p) + \frac1{m(p)} \le \va_{p}(z_p) + \frac1{m(p)},\ \forall p \ge j.
\end{aligned}
$$
By taking limsup in both sides as $p \to \infty$ and noting that $u_j=\va_{j, k(j)}$ on $\partial V$ we obtain
$$\va_{j, k(j)} (z^*) +\frac{\ep}2=\tilde u_j (z^*) +\frac{\ep}2 \le  \varlimsup\limits_{ p \to \infty} \va_{p} (z_p) \le u^* (z^*).$$
Here the last inequality follows from Lemma 2.7 and the fact that  $\va_{p} \downarrow u^*$ on $\ov V.$
This is  a contradiction to (\ref{inq5}).

Thus, our claim is fully proved.
Since $h_{p,j} \in PSH^c(V)$, it follows from Fornaess-Narasimhan's criterion that $u_j \in PSH^{-}_c (V)$. The proof is complete.
\end{proof}
\vskip1cm
\cen {\bf References}

\n
[Be]  E. Bedford, {\it The operator $(dd^c)^n$ on complex spaces}, S\'eminaire d'Analyse Lelong-Skoda, Lecture Notes in Math., {\bf 919} (1981), 294-324.

\n
[Ch] E. M. Chirka, {\it Complex Analytic Sets,} Kluwer, Dordecht, 1989.

\n
[De] J.-P. Demailly, {\it Mesures de Monge-Amp\`ere et caract\'erisation g\'eom\'etrique des vari\'et\'es
alg\'ebriques affines}, M\'em. Soc. Math. France (N.S.) {\bf 19} (1985), 1-124.

\n
[Ed] D. A. Edwards, {\it Choquet boundary theory for certain spaces of lower semicontinuous
functions}, in Function Algebras (Proc. Internat. Symposium on Function Algebras,
Tulane Univ., 1965) (Scott-Foresman, Chicago, 1966), pp. 300-309.

\n
[DW]  N.Q. Dieu and F. Wikstr\"om, {\it Jensen measures and approximation of plurisubharmonic functions}, Michigan Math. J. {\bf 53} (2005) 529-544.

\n
[Di] Nguyen Quang Dieu, {\it Approximation of plurisubharmonic functions on bounded domains in $\mathbb C^n$}, Michigan Math. J. {\bf 54}, (2006) 697-711.

\n
[FN] J. E. Fornaess and R. Narasimhan, {\it The Levi problem on complex spaces with singularities},
Math. Ann. {\bf 248} (1980) 47-72.

\n
[Kl] M. Klimek, {\it Pluripotential Theory}, Oxford 1991.

\n
[Si] N. Sibony, {\it Une classe de domaines pseudoconvexes}, Duke Math. J. {\bf 55} (1987), 299-319.

\n
[Wik1]  F. Wikstr\"om, {\it Jensen measures and boundary values of plurisubharmonic functions,}
Ark. Mat. {\bf 39} (2001) 181-200.

\n
[Wik2] F.  Wikstr\"om, {\it The Dirichlet problem for maximal plurisubharmonic functions on analytic varieties in $\C^n$}, International Journal of Mathematics
{\bf 20}, No. 4 (2009) 521-528.

\n
[Ze] A. Zeriahi, {\it  Fonction de Green pluricomplexe \`a pole \`a l'infini sur un espace de Stein parabolique et applications,}
Math. Scand. {\bf 69} (1991), 89-126.
\end{document}